\documentclass{amsart}

\usepackage[normalem]{ulem}
\usepackage{amsmath,amsthm,amsfonts,amscd,amssymb, tikz-cd} 
\usepackage{verbatim}      	
\usepackage{listings} 
\usepackage{mathtools}
\usepackage{epsfig}         	
\usepackage{url}		
\usepackage{epic,latexsym}
\usepackage{graphicx}
\usepackage{xcolor}

\usepackage{lcd}
\usepackage{mathrsfs}
\usepackage[centering,text={15.5cm,22cm},
        marginparwidth=20mm,dvips]{geometry}

\usepackage{hyperref}
\hypersetup{%
pdftitle={Preprint},
pdfauthor={},
pdfkeywords={},
bookmarksnumbered,
pdfstartview={FitH},
breaklinks=true,
colorlinks=true,
urlcolor=blue,
citecolor=blue,
linktocpage=true
}%
\usepackage[hypcap=true]{caption}

\usepackage{breakurl}

\usepackage{lscape}
\usepackage{tabularx}
\usepackage{marginnote}
\usepackage[all]{xy}
\usepackage{enumitem}
\usepackage{stmaryrd}
\usepackage{tipa}
\usepackage{upgreek}

\DeclareMathOperator{\tw}{tw}

\DeclareMathOperator{\N}{N}

\def\!{\mskip-\thinmuskip}
\let\?=\overline
\let\:=\colon

\newcommand{\pur}[1]{{\color{purple} #1}}

\theoremstyle{plain}
\ifdefined\thm
\else
	\theoremstyle{plain}
 	\newtheorem{thm}{Theorem}[section]
\fi

\ifdefined\prop
\else
	\theoremstyle{plain}
 	\newtheorem{prop}[thm]{Proposition}
\fi\ifdefined\cor
\else
	\theoremstyle{plain}
 	\newtheorem{cor}[thm]{Corollary}
\fi

\ifdefined\lem
\else
	\theoremstyle{plain}
 	\newtheorem{lem}[thm]{Lemma}
\fi

\ifdefined\def
\else
	\theoremstyle{definition}
	\newtheorem{def}[thm]{Definition}
\fi

\ifdefined\defn
\else
	\theoremstyle{definition}
	\newtheorem{defn}[thm]{Definition}
\fi

\ifdefined\example
\else
	\theoremstyle{definition}
    \newtheorem{example}[thm]{Example}
\fi

\ifdefined\eg
\else
		\theoremstyle{definition}
        \newtheorem{eg}[thm]{Example}
\fi

\ifdefined\rem
\else
		\theoremstyle{remark}
        \newtheorem{rem}[thm]{Remark}
\fi

  \newtheorem{dfn}[thm]{Definition}

\newcommand{\Z}{{\mathbb Z}}

\newcommand{\supp}{\operatorname{supp}}

\usepackage{tikz-cd}
\usepackage{multicol}

\newcommand{\KK}{{\text{KK}}}

\newcommand{\ufv}{\text{ex}}
\newcommand{\fv}{{\text{fr}}}

\newcommand{\bi}{\mathbf{i}}
\newcommand{\ubi}{\underline{\bi}}

\newcommand{\bh}{\mathbf{h}}
\newcommand{\ubh}{\underline{\bh}}

\newcommand{\NN}{\mathbb{N}}

\newcommand{\frC}{\mathfrak{C}}
\newcommand{\frc}{\mathfrak{c}}

\newcommand{\ta}{\widetilde{a}}
\newcommand{\tb}{\widetilde{b}}

\renewcommand{\tw}{\operatorname{tw}}

\ifdefined\supp
\else
\newcommand{\supp}{\operatorname{supp}}
\fi

\ifdefined\tw
\else
	\newcommand{\tw}{{\opname tw}}
\fi

\ifdefined\G
    \renewcommand{\G}{{\mathbb G}}
\else
   \newcommand{\G}{{\mathbb G}}
\fi

\ifdefined\C
    \renewcommand{\C}{{\mathbb C}}
\else
   \newcommand{\C}{{\mathbb C}}
\fi

\usepackage{comment}

\title{From $\mathbf{i}$-boxes to signed words}

\author{Alessandro Contu}
\address{Universit\'{e} Paris Cit\'{e}\\ France}
\email{alessandro20.contu@gmail.com}

\author{Fan Qin}
\address{Beijing Normal University \\ China}
\email{qin.fan.math@gmail.com}

\author{Qiaoling Wei}
\address{Capital Normal University \\ China}
\email{wql03ster@gmail.com}

\thanks{}

\begin{document}

        \begin{abstract}	
The combinatorics of $\mathbf{i}$-boxes has recently been introduced by Kashiwara--Kim--Oh--Park in the study of cluster algebras arising from the representation theory of quantum affine algebras. In this article, we associate to each chain of $\mathbf{i}$-boxes a signed word, which canonically determines a cluster seed following Berenstein--Fomin--Zelevinsky. By bridging these two different languages, we are able to provide a quick solution to the problem of explicit determining the exchange matrices associated with chains of $\mathbf{i}$-boxes. 
\end{abstract}

\maketitle
 
	\setcounter{tocdepth}{1} \tableofcontents{}
	
\section{Introduction}	\label{intro}

\subsection{Backgrounds}
Cluster algebras were introduced by Fomin and Zelevinsky \cite{fomin2002cluster}. These algebras possess distinguished elements called cluster monomials. One of the main reasons for the interest in the theory of cluster algebras is their unexpected emergence across diverse areas of mathematics. An illustrative case is the representation theory of quantum affine algebras, where cluster algebra structures are studied through the framework of monoidal categorification, cf. the inspiring work of Hernandez--Leclerc \cite{HernandezLeclerc09}. A monoidal category $\mathscr{C}$ is a monoidal categorification of a cluster algebra $\mathcal{A}$ if there is an isomorphism between the Grothendieck ring of $\mathscr{C}$ and $\mathcal{A}$, such that the cluster monomials of $\mathcal{A}$ correspond to certain simple objects of $\mathscr{C}$. 

Let $\mathfrak{g}$ be a simple complex finite-dimensional Lie algebra. In 2023, Kashiwara--Kim--Oh--Park \cite{KKOP_2024_monoidalII} defined certain monoidal Serre subcategories $\mathscr{C}^{[a,b],\mathcal{D}_\mathcal{Q},\underline{w}_0}$ of the module category of the quantum affine algebra of $\mathfrak{g}$, where $[a,b]$ denotes a possibly unbounded integer interval (for more details, see \cite[\S 4, \S 6]{KKOP_2024_monoidalII}). In order to show that these categories are instances of monoidal categorification, Kashiwara--Kim--Oh--Park introduce the combinatorics of \emph{chains of} $\mathbf{i}$-\emph{boxes} (cf. \cite[\S 4, \S 5]{KKOP_2024_monoidalII}). A chain of $\mathbf{i}$-boxes is a sequence of integer intervals satisfying certain technical conditions (see Section \ref{section_iboxes}). In particular, the definition of $\mathbf{i}$-boxes is based on the choice of an infinite sequence of indexes $\mathbf{i}$. Let us fix a category $\mathscr{C}^{[a,b],\mathcal{D}_\mathcal{Q},\underline{w}_0}$. In this setting, the sequence $\mathbf{i}$ is of the form $\underline{\widehat{w}}_0$, a particular sequence whose elements belong to the index set of a simply laced Dynkin diagram canonically associated to $\mathfrak{g}$. Kashiwara--Kim--Oh--Park associate an affine determinant module $M(\mathfrak{c})$ (a generalization of the Kirillow-Reshetikhin modules) to each $\mathbf{i}$-box $\mathfrak{c}$. Moreover, for each chain of $i$-boxes $\mathfrak{C}=(\mathfrak{c}_i)$, they show the existence of a skew-symmetric exchange matrix $B(\mathfrak{C})$ such that, together with the representatives $([M(\mathfrak{c})_i)])_i$ of the modules associated to the $\mathbf{i}$-boxes of the chain, they give a seed for a cluster algebra structure of the Grothendieck ring of $\mathscr{C}^{[a,b],\mathcal{D}_\mathcal{Q},\underline{w}_0}$ (cf. \cite[Thm. 8.1]{KKOP_2024_monoidalII}). 

More precisely, when $b$ is an integer, Kashiwara--Kim--Oh--Park start by explicitly providing the exchange matrix $B(\mathfrak{C}^{[a,b]}_-)$ associated to a specific chain of $i$-boxes, denoted $\mathfrak{C}^{[a,b]}_-$ (see Definition \ref{def:initial-box}), generalizing a construction of Hernandez--Leclerc \cite{HernandezLeclerc11}. They show that the matrix $B(\mathfrak{C})$ can be obtained from the matrix $B(\mathfrak{C}^{[a,b]}_-)$ through a sequence of mutations. The case $b=\infty$ is treated through a limit procedure. At the end of \cite{KKOP_2024_monoidalII}, they state the problem of finding an explicit formula for all the matrices $B(\mathfrak{C})$. See Remark \ref{rem:problem-monoidal} for more details of this problem in terms of monoidal categories.

This problem is natural and fundamental for understanding the cluster structures appearing in the representation theory of quantum affine algebras. Recently, two different solutions have been proposed:
\begin{itemize}
    \item The first author (cf. \cite{Contu_solutionmonoidalbyadditive}) proposed a solution translating the problem in a framework of additive categorification. Each exchange matrix $B(\mathfrak{C})$ is obtained as a sub-matrix of a square matrix $\overline{B}(\mathfrak{C})$,  which, starting from the explicitly given matrix $\overline{B}(\mathfrak{C}^{[a,b]}_-)$, can be expressed via a matrix multiplication (Palu's generalized mutation rule \cite{Palu08}): 
    \[ \overline{B}(\mathfrak{C}) = P \overline{B}(\mathfrak{C}^{[a,b]}_-) P^{-t},\]
    where $P$ is an invertible matrix obtained through the computation of indices of cluster-tilting objects of a suitable cluster category.
   
   \item In \cite{kashiwarakim2024exchangematricesiboxes}, Kashiwara--Kim work with sequences $\mathbf{i}$ taking values in the index set of a generalized Cartan matrix $C$ and with \emph{maximal commuting family} of $\mathbf{i}$-boxes. Each chain of $\mathbf{i}$-boxes forms a maximal commuting family. For any such family $\mathcal{F}$, they define an $\mathcal{F}\times\mathcal{F}$-skew-symmetrizable matrix $\widetilde{B}^\KK(\mathcal{F})$ (see \S \ref{subsection_comparison_matrices}). When $C$ is a symmetric Cartan matrix of finite type and $\mathbf{i}$ is of the form $\underline{\widehat{w}}_0$, each exchange matrix $B(\mathfrak{C})$ can be obtained as a sub-matrix $B^\KK(\mathfrak{C})$ of the matrix $\widetilde{B}^\KK(\mathfrak{C})$.
\end{itemize}
Notice that the solution to Kashiwara--Kim--Oh--Park's problem provided by the first-named author, although interesting for bridging monoidal and additive categorification of cluster algebras, is not as explicit, since it requires to perform a multiplication of matrices. Kashiwara--Kim provide a direct formula relying on monoidal categorification and elaborate combinatorial machinery.

\subsection{Main results}
In this work, we propose an alternative and straightforward solution to Kashiwara--Kim--Oh--Park's problem. The starting point is the combinatorics and the formalism of signed words. Recall that a signed word on an index set $I$ is a sequence whose elements are of the form $\varepsilon h$, where $\varepsilon$ is in $\{\pm 1\}$ and $h$ is in $I$. Assume that $I$ is the index set of a generalized Cartan matrix. To each signed word $\ubh$, one can associate a seed $\mathbf{t}(\ubh)$ following \cite{BerensteinFominZelevinsky05}, which plays an important role in studying the cluster structures on double Bott-Samelson cells \cite{shen2021cluster}\cite{qin2024infinite}. Let $B(\ubh)$ be the corresponding exchange matrix.
Assume that $\mathbf{i}$ takes value in $I$ and that $b$ is in $\mathbb{Z}$. For each chain of $\mathbf{i}$-boxes $\mathfrak{C}$ associated to $I$, we define algorithmically a skew-symmetrizable matrix $B(\mathfrak{C})$ in a similar fashion to \cite{KKOP_2024_monoidalII}. Moreover, using the indices of the $\mathbf{i}$-boxes of the chain, we define a signed word $\ubh(\mathfrak{C})$. Our main result states that the desired matrix  $B(\mathfrak{C})$ is given by the well-known matrix $B(\ubh(\mathfrak{C}))$.
\begin{thm}\label{thm:box-to-word}
    The matrix $B(\mathfrak{C})$ equals $B(\ubh(\mathfrak{C}))$.
\end{thm}
Therefore, by applying our main result to the setting where the Cartan matrix is symmetric of finite type and $\mathbf{i}$ is of the form $\underline{\widehat{w}}_0$, we obtain a solution to the problem of Kashiwara–Kim–Oh–Park, via a translation from the combinatorics of signed words to that of $\mathbf{i}$-boxes.

The trace of the proof of Theorem \ref{thm:box-to-word} is the following:
\begin{enumerate}
    \item By direct comparison, we verify that the matrices $B(\mathfrak{C}^{[a,b]}_-)$ and $B(\ubh(\mathfrak{C}^{[a,b]}_-))$ are equal.
    \item Let $\mathfrak{C}'$ and $\mathfrak{C}$ be any chains related by a box move. We show that, if $B(\mathfrak{C})$ equals $B(\ubh(\mathfrak{C}))$, then $B(\mathfrak{C}')$ also equals $B(\ubh(\mathfrak{C}'))$.
\end{enumerate}

 Additionally, when $I$ is the set of indices of a generalized Cartan matrix, we directly verify that our matrix $B(\ubh(\mathfrak{C}))$ corresponds to Kashiwara--Kim's matrix $B^\KK(\mathfrak{C})$ (Proposition \ref{prop_matrices_equal}).

 Note that Theorem \ref{thm:box-to-word} implies that the matrix $B(\mathfrak{C})$ is independent of the choice of box moves from $\frC_-^{[a,b]}$ to $\frC$ (Corollary \ref{cor:indepedent}). It also determines matrices in the case $b=+\infty$, as colimits of the matrices in the cases $b\in \Z$ (Corollary \ref{cor:extension}, see also \cite{qin2024infinite}).

\subsection{Notations and conventions}\label{sec:notations}

Choose any finite subsets $K^{\ufv}\subset K$. For any $K\times K^{\ufv}$-matrix $Z=(Z_{ij})$ and any permutation $\sigma$ on $K$, we define the $K\times \sigma(K^{\ufv})$-matrix $\sigma Z$ such that $(\sigma Z)_{\sigma i,\sigma j}:=Z_{i,j}$. Let $\sigma_{j,j+1}$ denote the transposition $(j,j+1)$.

The matrix $Z$ is called an exchange matrix if $Z_{ik}\in \Z$ for $(i,k)\in K\times K^{\ufv}$ and, moreover, it is skew-symmetrizable, i.e., there exists a diagonal matrix $D=(D_{kk})_{k\in K^{\ufv}}$ with diagonal entries $D_{kk}\in \N_{>0}$, called a skew-symmetrizer, such that $D_{ii}Z_{ik}=-D_{kk}Z_{ki}$, $\forall i,k\in K^{\ufv}$.

Let $[\ ]_+$ denote $\max(\ ,0)_+$. Let $Z$ denote an exchange matrix. Following \cite{fomin2002cluster}, for any $k\in K^{\ufv}$, the mutation $\mu_k$
gives us a new exchange matrix $\mu_k Z$ such that
\[(\mu_k Z)_{ij}=\begin{cases}
        Z_{ij}+[Z_{ik}]_+[Z_{kj}]_+-[-Z_{ik}]_+[-Z_{kj}]_+ & \text{if }\ i,j\neq k,\\
        -Z_{ij} & \text{if }\ i=k\ \text{or }j=k.
    \end{cases}\]

\section{Combinatorics of $\mathbf{i}$-box}
\label{section_iboxes}
In this section, following \cite{KKOP_2024_monoidalII,kashiwarakim2024exchangematricesiboxes}, we recall the definition and the properties of $\mathbf{i}$-boxes.\\
For $a,b\in\mathbb{Z}\sqcup\{\pm\infty\}$,  we write $[a,b]$ for the integer interval
\[ [a,b]=\{k\in\mathbb{Z} \ |\ a\leq k\leq b\}.\]
The \emph{length} $l$ of an integer interval $[a,b]$ is defined as $l=\mathrm{max}(b-a+1,0)$.

Let $I$ be a finite set of indices and  $Z$ be an integer interval. We write $\mathbf{i}=(i_Z)_{k\in Z}$ for a sequence of elements of $I$ indexed by the elements of $Z$. For $s\in Z$ and $j\in I$, we will use the following notations:
\begin{align}\label{eq:i-word-symbols}
\begin{split}
s^+& =\text{min}(\{t\in Z\ |\ s<t,\ i_t=i_s\}\cup \{\infty\}), \quad   s^-=\text{max}(\{t\in Z\ |\ s>t,\ i_t=i_s\}\cup \{-\infty\}),\\
s(j)^\oplus&=\text{min}(\{t\in Z\ |\ s\leq t,\ i_t=j\}\cup \{\infty\}), \quad  s(j)^\ominus=\text{max}(\{t\in Z\ |\ s\geq t,\ i_t=j\}\cup \{-\infty\}).
\end{split}
\end{align}

\begin{dfn}[{\cite[\S 2.1]{kashiwarakim2024exchangematricesiboxes}\cite[\S 6.1]{qin2023analogs}}]\mbox{}\label{dfn:ibox}

A finite integer interval $[a,b]$ in $Z$ is an {\em $\mathbf{i}$-box} if $i_a=i_b$. 
For an $\mathbf{i}$-box $[a,b]$, we define its {\em color} $i([a,b])$ as $i([a,b])=i_a=i_b$ and its {\em $i$-cardinality} or {\em order} as the number of times that the index $i([a,b])$ appears in the sub-interval of $\mathbf{i}$ corresponding to $[a,b]$.
\end{dfn}

\begin{rem}
 The intervals in Definition \ref{dfn:ibox} are closely related to Kirillov-Reshetikhin modules of quantum affine algebras. They were called $\mathbf{i}$-boxes by \cite{KKOP_2024_monoidalII}, which focused on the case $I=I_\mathsf{g}$, $\mathbf{i}=\widehat{\underline{w}}_0$ and $Z=\mathbb{Z}$, where $I_\mathsf{g}$ is the set of Dynkin indices of a simply-laced Lie algebra $\mathsf{g}$ and $\widehat{\underline{w}}_0=(i_k)_{k\in \mathbb{Z}}$ is an infinite sequence obtained from a reduced expression $\underline{w}_0=s_{i_1}\dots s_{i_l}$ of the longest element of the Weyl group of $\mathsf{g}$ by extending the sequence $i_1,\dots,i_l$ via the rule
\[ i_{k+l}=i_k^*,\]
where $(-)^*$ is the involution on the index set $I_\mathsf{g}$ induced by $w_0(\alpha_i)=-\alpha_{{i}^*}$, for any simple root $\alpha_i$, $i \in I_\mathsf{g}$.
\end{rem}

For a finite interval $[a,b]$ in $Z$, we define the $\mathbf{i}$-boxes
\[ [a,b\}=[a,b(i_a)^\ominus] \text{  and  } \{a,b]=[a(i_b)^\oplus,b]. \]
In other terms, $[a,b\}$ and $\{a,b]$ are the largest $\mathbf{i}$-boxes contained in $[a,b]$ with color $i_a$ and $i_b$ respectively.
 When we want to emphasize that an $\mathbf{i}$-box is of color $j$, we use the notation $[a,b]_j$.

\begin{dfn}[{\cite[Def. 5.1]{KKOP_2024_monoidalII}}]
Let $l$ be in $\mathbb{N}\cup \{\infty\}$. A \textit{chain} of $\mathbf{i}$-boxes of \emph{length} $l$ is a sequence of $\mathbf{i}$-boxes $\mathfrak{C}=(\mathfrak{c}_k)_{1\leq k< l+1}$ satisfying the following conditions for any $1\leq s< l+1$:
\begin{itemize}
    \item[(i)] The union $[\widetilde{a}_s,\widetilde{b}_s]:=\bigcup_{1\leq k\leq s} \mathfrak{c}_k $ is an interval of length $s$;
    \item [(ii)] The $\mathbf{i}$-box $\mathfrak{c}_s$ is the largest $\mathbf{i}$-box of color $i(\mathfrak{c}_s)$ contained in the interval $\bigcup_{1\leq k\leq s} \mathfrak{c}_k $.
\end{itemize}
Note that condition (ii) implies $b_s\leq \tb_s<b_s^-$ and $a_s\geq \ta_s>a_s^-$. In addition, we have $[\ta_s,\tb_s]\subset [\ta_t,\tb_t]$ whenever $s<t$. We call the interval $\bigcup_{1\leq k< l+1} \mathfrak{c}_k $ the \emph{range} of the chain. For any  $1\leq s < l+1$, the sequence $(\mathfrak{c}_k)_{1\leq k\leq s}$ is a chain of $\mathbf{i}$-boxes, called a \textit{sub-chain} of $\mathfrak{C}$.
\end{dfn}

\begin{rem}[{\cite[\S 5]{KKOP_2024_monoidalII}}]
To each chain of $\mathbf{i}$-boxes $\mathfrak{C}=(\mathfrak{c}_k)_{1\leq k< l+1}$, we can bijectively associate a pair $(c, (E_k)_{1\leq k<l})$, where $c\in Z$ and each $E_k$ is a symbol in  $\{L,R\}$, in the following way: let $c$ be the integer such that $\{c\}=\mathfrak{c}_1$ and, for any $1\leq k<l$, set
\[
E_k=\begin{cases}
    L, & \text{if }  [\widetilde{a}_{k+1},\widetilde{b}_{k+1}]=[\widetilde{a}_k-1,\widetilde{b}_k],\\ 
    R, & \text{otherwise.}
\end{cases}
\]
In fact, given such a pair $(c, (E_k)_{1\leq k<l})$, the associated chain of $\mathbf{i}$-boxes $\mathfrak{C}$ can be recursively recovered as follows:
\begin{itemize}
    \item $\mathfrak{c}_1=\{c\}$;
    \item  for any $2\leq k<l+1$, we have
    \[ 
    \mathfrak{c}_k=\begin{cases}
        [\widetilde{a}_{k-1}-1,\widetilde{b}_{k-1}\}, &\text{if } E_{k-1}=L;\\
        \{\widetilde{a}_{k-1},\widetilde{b}_{k-1}+1], &\text{if } E_{k-1}=R.
    \end{cases}
    \]
\end{itemize}
We refer to $T=L$ (resp. $=R)$ as a \emph{left} (resp. \emph{right}) \emph{expansion operator} and to a sequence $(c, (E_k)_{1\leq k<l})$ as a \emph{rooted sequence of expansion operators}, see \cite{Contu_solutionmonoidalbyadditive}.

\end{rem}

\begin{dfn}\label{def:initial-box}
    Let $[a,b]$ be an integer interval with $a\leq b$, $a\in \mathbb{Z}\sqcup\{-\infty\}$ and $b\in \mathbb{Z}$. Denote its length by $l=b-a+1$. Following \cite{KKOP_2024_monoidalII}, we define $\mathfrak{C}_-^{[a,b]}$ as the chain of $\mathbf{i}$-boxes associated to the rooted sequence of expansion operators $(b,(E_k)_{1\leq k<{l}})$, where $E_k=L$ for any $k$. Explicitly, the $k$-th $\mathbf{i}$-boxes are $\mathfrak{c}_k=[b-k+1,b\}$, $\forall k\in [1,l]$.
\end{dfn}

\begin{dfn}[{\cite[\S 5]{KKOP_2024_monoidalII}}]
\label{de_box_move}
Let $\mathfrak{C}=(\mathfrak{c}_k)$ be a chain of $\mathbf{i}$-boxes of length $l\leq \infty$ corresponding to a pair $(c,(E_k)_{1\leq k<l})$. 
\begin{itemize}
\item[(i)] For $1 \leq s < l$, the $\mathbf{i}$-box $\mathfrak{c}_s$ is defined to be \textit{movable} if $s=1$ or  $s\geq 2$ and $E_{s-1}\neq E_s$.
\item[(ii)] For a movable $\mathbf{i}$-box $\mathfrak{c}_s$, the \textit{box move} at $s$, denoted by $\nu_s$, is the operation sending $\mathfrak{C}$ 
to the chain $\nu_s\mathfrak{C}$, whose associated pair $(c',E')$ is defined as follows:
\[
c'=\begin{cases}
c+1 & \text{if } s=1, E_1=R,\\
c-1 & \text{if } s=1, E_1=L,\\
c & \text{if } s>1,
\end{cases} \ \text{   and   }\
E'_k=\begin{cases}
    R & \text{if } E_k=L,\ k\in\{s-1,s\},\\
    L & \text{if } E_k=R,\ k\in\{s-1,s\},\\
    E_k & \text{if } k\notin\{s-1,s\}.
\end{cases}
\]
\item[(iii)] We call a finite composition of box moves a \textit{chain transformation}.
\end{itemize}
\end{dfn}

\begin{eg}
Let $I=\{1,2,3\}$ be the set of Dynkin indices of a simple Lie algebra of type $A_3$. Let $\underline{w}_0=s_1s_2s_3s_1s_2s_1$ be a reduced expression of the longest element Weyl group of type $A$ and let $\mathbf{i}$ be the sequence $\underline{\widehat{w}}_0$:
\[ \widehat{\underline{w_0}}=\dots \underbrace{ 1,\ 3,\ 2,\ 3,\ 1,\ 2,\ 3,\ 1}_{[-3,4]} \  
 \dots .\]
 The chain of $i$-boxes $\mathfrak{C}=(\mathfrak{c}_k)_{1\leq k\leq 8}$ of range $[-3,4]$ associated to the rooted sequence of expansion operators $(4,(L,L,L,L,L,L,L))$ is given by
  \begin{align*}
\mathfrak{c}_1&=[4]_1, &  \mathfrak{c}_2&=[3]_3, &\mathfrak{c}_3&=[2]_2, &\mathfrak{c}_4&=[1,4]_1,\\
\mathfrak{c}_5&=[0,3]_3, &  \mathfrak{c}_6&=[-1,2]_2, &\mathfrak{c}_7&=[-2,3]_3 &\mathfrak{c}_8&=[-3,4]_1.
\end{align*} 

Notice that the only movable $i$-box in the chain $\mathfrak{C}$ is $\mathfrak{c}_1$. The box move at 1 sends $\mathfrak{C}$ to the chain $\nu_1\mathfrak{C}=(\mathfrak{c}_k')_{1\leq k\leq 6}$ associated to the sequence $(3,(R,L,L,L,L,L,L))$:

  \begin{align*}
\mathfrak{c}'_1&=[3]_3, &  \mathfrak{c}'_2&=[4]_1, &\mathfrak{c}'_3&=[2]_2, &\mathfrak{c}'_4&=[1,4]_1,\\
\mathfrak{c}'_5&=[0,3]_3, &  \mathfrak{c}'_6&=[-1,2]_2, &\mathfrak{c}'_7&=[-2,3]_3 &\mathfrak{c}'_8&=[-3,4]_1.
\end{align*} 

Notice that the $\mathbf{i}$-box $\mathbf{c}'_2$ is movable and that the associated box move sends $\nu_1\mathfrak{C}$ to the chain to the chain associated to the sequence $(3,(L,R,L,L,L,L,L)))$. Iterating this process, we obtain that, through a composition of box moves, the chain $\mathfrak{C}$ is sent to the chain $\widetilde{\mathfrak{C}}=(\widetilde{\mathfrak{c}_k})_{1\leq k\leq 8}$ associated to the sequence $(3,(L,L,L,L,L,L,R))$:
  \begin{align*}
\widetilde{\mathfrak{c}}_1&=[3]_3, &  \widetilde{\mathfrak{c}}_2&=[2]_2, &\widetilde{\mathfrak{c}}_3&=[1]_1, &\widetilde{\mathfrak{c}}_4&=[0,3]_3,\\
\widetilde{\mathfrak{c}}_5&=[-1,2]_2, &  \widetilde{\mathfrak{c}}_6&=[-2,3]_3, &\widetilde{\mathfrak{c}}_7&=[-3,1]_1 &\widetilde{\mathfrak{c}}_8&=[-3,4]_1.
\end{align*} 
\end{eg}

\begin{rem}[{\cite[Lemma 5.10]{KKOP_2024_monoidalII}, \cite[Rem. 2.10]{Contu_solutionmonoidalbyadditive}}]
    Let $[a,b]$ be an integer interval with $b\in \mathbb{Z}$ and $a\in \mathbb{Z}\sqcup \{-\infty\}$. Then any two chains of $\mathbf{i}$-boxes of range $[a,b]$ are related by a chain transformation.
\end{rem}

\begin{rem}
\label{rem_color_exchange_op}
Let $\mathfrak{C}$ and $\mathfrak{C}'$ be two chain of $\mathbf{i}$-boxes. Assume that $\mathfrak{C}$ and $\mathfrak{C}'$ are related by a box move at $s\geq 1$. Then,

\begin{itemize}
    \item if $s\geq 2$, we have: 
\[
i(\frc_{s+1})=i(\frc'_{s}), \quad i(\frc_{s})=i(\frc'_{s+1})\quad \text{and} \quad i(\frc_{k})=i(\frc'_{k})\ \text{ for any } k\notin \{s,s+1\}. 
\]\
\item if $s=1$, we have 
\[ i(\frc_2)=i_{\frc'_1}, \quad i(\frc'_2)=i_{\frc_1}\quad \text{and} \quad i(\frc_{k})=i(\frc'_{k})\ \text{ for any } k\geq 3  \]
\end{itemize}

\end{rem}

\begin{dfn}[{\cite[Def. 2.13]{kashiwarakim2024exchangematricesiboxes}}]
Let $\mathfrak{C}=(\mathfrak{c}_k)_k$ be a chain of $\mathbf{i}$-boxes with associated sequence of extension operators $(E_k)_k$. For any $k$, the \emph{effective end} $z$ of the $\mathbf{i}$-box $\mathfrak{c}_k=[x,y]$ is defined as
\[
z=\begin{cases}
    y & \text{if } k=1 \text{ or } E_{k-1}=R,\\
    x & \text{if } k=1 \text{ or } E_{k-1}=L. 
\end{cases}
\]
\end{dfn}

\subsection{The matrix associated to a chain of $\mathbf{i}$-boxes} 
From now on, let $I$ be the set of indices of a generalized Cartan matrix $C=(c_{ij})_{i,j\in I}$. Let $[a,b]$ be an integer interval and let $l$ be its length. Let $\mathfrak{C}=(\mathfrak{c}_k)_{k\in [1,l]}$  be a chain of $\mathbf{i}$-boxes of range $[a,b]$. In this subsection, following \cite{KKOP_2024_monoidalII}, under the assumption that $b$ is an integer, we recursively associate to the chain $\mathfrak{C}$ an exchange matrix $B(\mathfrak{C})$. 

We introduce the following sets 
\begin{align*}
K(\mathfrak{C})&=\begin{cases}
    \{1,\dots, l\}, &\text{ if } l<\infty;\\
    \mathbb{N}_{\geq 1}, &\text{ if } l=\infty;
\end{cases}\\
K(\mathfrak{C})^{\text{fr}}&=\{s \in K(\mathfrak{C})\ |\ \mathfrak{c}_s=[a(i)^\oplus,b(i)^\ominus] \text{ for some } i \in I \};\\
K(\mathfrak{C})^{\text{ex}}&=K(\mathfrak{C})\backslash K(\mathfrak{C})^{\text{fr}}.
\end{align*}

For any $k\in K(\mathfrak{C})$, we define $k[1]$ as 
\[ k[1]=\mathrm{min}( \{k'\in [k+1,l]\ |\ \mathfrak{c}_{k'} \text{ has the same color of } \mathfrak{c}_k \}\sqcup \{+\infty\}.\]

Assume that $b$ is in $\mathbb{Z}$. Following \cite[\S 7.5]{KKOP_2024_monoidalII}, let $B(\mathfrak{C}_-^{[a,b]})=(b_{jk})_{j\in K(\mathfrak{C}),k\in K(\mathfrak{C})^{\text{ex}}}$ be the exchange matrix given by 

\[b_{jk}=
\begin{cases}
    1 & k=j[1],\\
    -1 & j=k[1],\\
    c_{i_{j},i_{k}}& j<k<j[1]<k[1],\\
    -c_{i_{j},i_{k}}& k<j<k[1]<j[1],\\
    0,& \text{otherwise.}
\end{cases}
\]
Although this construction was given in \cite{KKOP_2024_monoidalII} only for Cartan matrices of ADE type (see also Remark 2.12), we apply the same formula in the more general setting considered here.

Next, for any chain of $\mathbf{i}$-boxes $\mathfrak{C}=(\mathfrak{c}_k)_{k\in [1,l]}$ of range $[a,b]$, fix a sequence of box moves $\nu_1,\dots,\nu_N$ whose composition sends $B(\mathfrak{C}_-^{[a,b]})$ to $B(\mathfrak{C})$. For $0\leq s\leq N$, we write $\mathfrak{C}_s=(\mathfrak{c}^s_k)_{k\in [1,l]}$ for the chain of $i$-boxes \[\mathfrak{C}_s=\begin{cases} 
    \mathfrak{C}_-^{[a,b]} & \text{if } s=0,\\
    \nu_s\circ\dots\circ \nu_1 \mathfrak{C}_-^{[a,b]}& \text{otherwise }.
    \end{cases}
    .\]

Let $B_0$ be the matrix $B(\mathfrak{C}_-^{[a,b]})$ and, for any $1\leq s\leq N$, recursively define the exchange matrix $B_s$ as follows:

\begin{itemize}
    \item[(i)] If the $\mathbf{i}$-box $\mathfrak{c}^{s-1}_{s+1}$ has the same color as $\mathfrak{c}^{s-1}_s$, then
        \[B_s= \mu_s(B_{s-1}).\]

    \item[(ii)] If the $\mathbf{i}$-boxes $\mathfrak{c}^{s-1}_{s+1}$ and $\mathfrak{c}^{s-1}_s$ have different color, then 
    \[B_s= \sigma_{s,s+1}(B_{s-1}).\]
    
    \end{itemize}
Finally, we set $B(\mathfrak{C})=B_N$. In the next section, we will show that the matrix $B(\mathfrak{C})$ does not depend on the choice of the composition of box-moves sending $\mathfrak{C}_-^{[a,b]}$ to $\mathfrak{C}$

\begin{rem}\label{rem:problem-monoidal} Let $\mathfrak{g}$ be a finite-dimensional simple complex Lie algebra and let $\mathsf{g}$ be the Lie algebra of of simply-laced type associated to the unfolding of the Dynkin diagram of $\mathfrak{g}$ (cf. \cite{Fujita_Oh_2021}).
Assume that $b$ is in $\mathbb{Z}\cup\{\infty\}$ and that $\mathbf{i}=(i_k)_{k\in [a,b]}$ is a subsequence of the sequence $\widehat{\underline{w}}_0=(i_k)_{k\in \mathbb{Z}}$, for a certain reduced expression $\underline{w}_0$ of the longest element of the Weyl group of $\mathsf{g}$. In \cite{KKOP_2024_monoidalII}, Kashiwara--Kim--Oh--Park associate a \emph{monoidal seed}  (see \cite[Def. 7.2]{KKOP_2024_monoidalII} for the terminology) to each chain of $i$-boxes $\mathfrak{C}$. In particular, they define a family of \emph{commuting real prime} modules $M(\mathfrak{C})$ and show the existence of a companion exchange matrix $B(\mathfrak{C})$. In order to do this, when $b$ is an integer they use the above procedure (cf. \cite[Prop. 8.11]{KKOP_2024_monoidalII}) and, when $b=+\infty$ they define $B(\mathfrak{C})$ as a colimit of the matrices associated to the subchain of $\mathfrak{C}$ (by \cite[Prop. 7.14, Lem. 7.16]{KKOP_2024_monoidalII}, this operation is well defined). Moreover, it follows from their theory (cf. \cite[Prop. 7.14]{KKOP_2024_monoidalII}) that the exchange matrix $B(\mathfrak{C})$ does not depend on the choice of the sequence of box moves sending $\mathfrak{C}_{-}^{[a,b]}$ to $\mathfrak{C}$. Nevertheless, they do not provide an explicit description of the coefficients of $B(\mathfrak{C})$, 
whose determination is stated as an open problem.

\end{rem}

\section{Signed words}

Recall that $C=(c_{ij})_{i,j\in I}$ is a generalized Cartan matrix. Let $l$ be in $\NN \sqcup \{+\infty\}$.
\begin{dfn}
A signed word on the index set $I$ is a sequence $\ubh=(\varepsilon_k h_k)_{1\leq k < l+1}$ such that, for any $k$, $\varepsilon_k\in \{\pm 1\}$ and $h_k\in I$. Denote $\bh_k=\varepsilon_k h_k$ and $|\bh_k|=i_k$.
\end{dfn}

We introduce the following sets 
\begin{align*}
K(\ubh)&=[1,l];\\
K(\ubh)^{\ufv}&=\{s \in K(\ubh)\ |\ \exists t\in [s+1,l],\ h_t=h_s\};\\
K(\ubh)^{\fv}&=K(\mathfrak{C})\backslash K(\ubh)^{\ufv}.
\end{align*}
For any $k\in K(\ubh)$, we define $k[1]$ as 
\begin{align}\label{eq:interval-shift}
    k[1]=\mathrm{min}( \{k'\in [k+1,l]\ |\ |h_{k'}|=|h_k| \}\sqcup \{+\infty\}.
\end{align} 
Note that $k[1]$ in \eqref{eq:interval-shift} and $s^+$ in \eqref{eq:i-word-symbols} should not be confused.

Following \cite{BerensteinFominZelevinsky05} \cite{shen2021cluster} \cite[(6.1)]{qin2023analogs}, for any signed word $\ubh$,
we define $\widetilde{B}(\ubh)=(\widetilde{b}_{jk})_{j,k\in K(\ubh)}$ as
\begin{align}
\widetilde{b}_{jk} & =\begin{cases}
\varepsilon_{k} & k=j[1]\\
-\varepsilon_{j} & j=k[1]\\
\varepsilon_{k}c_{|\bh_{j}|,|\bh_{k}|} & \varepsilon_{j[1]}=\varepsilon_{k},\ j<k<j[1]<k[1]\\
\varepsilon_{k}c_{|\bh_{j}|,|\bh_{k}|} & \varepsilon_{k}=-\varepsilon_{k[1]},\ j<k<k[1]<j[1]\\
-\varepsilon_{j}c_{|\bh_{j}|,|\bh_{k}|} & \varepsilon_{k[1]}=\varepsilon_{j},\ k<j<k[1]<j[1]\\
-\varepsilon_{j}c_{|\bh_{j}|,|\bh_{k}|} & \varepsilon_{j}=-\varepsilon_{j[1]},\ k<j<j[1]<k[1]\\
0 & \text{otherwise}
\end{cases}.\label{eq:dBS_B_matrix}
\end{align}
It has a skew-symmetrizer $\widetilde{D}(\ubh)$, which is a diagonal matrix with diagonal entries $\widetilde{D}_{ii}=h_i$. Let $B(\ubh)$ denote its $K(\ubh)\times K(\ubh)^{\ufv}$-submatrix.

From now on, we will always assume that $B(\ubh)$ is a locally finite matrix, i.e., for any $j$, only finitely many $b_{jk}$ are nonzero and, for any $k$, only finitely many $b_{jk}$ are nonzero. This assumption allows us to extend results in \cite{shen2021cluster} \cite{qin2023analogs} for $l\in \NN$ to the case $l=+\infty$ as in \cite{qin2024infinite}. 

Let $\ubh_{[j,k]}$ denote the sequence $(\bh_s)_{s\in [j,k]}$.
\begin{dfn}[{\cite[Section 2.3]{shen2021cluster},\cite[Section 3.2]{qin2024infinite}}]
    Let $\ubh=(\ubh_k)_{k\in [1,l]}$ be a signed word.
\begin{itemize}
    \item The \emph{left reflection} of $\ubh$ is the operation sending $\ubh$ to the signed word $\ubh'=(-\bh_1,\ubh_{[2,l]})$.
        \item Let $j$ be in $[1,l-1]$ such that $\ubh_j$ and $\ubh_{j+1}$ have different signs. Then the \emph{flip} of $\ubh$ at $j$ is the operation sending $\ubh$ to the signed word $\ubh'=(\ubh_{[1,j-1]},\bh_{j+1},\bh_j,\ubh_{[j+2,l
        ]})$ 
    \end{itemize}
\end{dfn}

\begin{prop}[{\cite[Prop. 3.7]{shen2021cluster},\cite[Section 3.2]{qin2024infinite}}]
Let $\ubh=(\ubh_k)_{k\in [1,l]}$ and $\ubh'=(\ubh_k')_{k\in [1,l]}$ be two signed words. 
\begin{enumerate}
    \item If $\ubh'$ is obtained from $\ubh$ via a left reflection, then $B(\ubh')=B(\ubh)$.
    \item If $\ubh'$ is obtained from $\ubh$ via a flip at $j$, then 
    \[B(\ubh')=\begin{cases}
        \sigma_{j,j+1}B(\ubh) & \text{if }\ h_j\neq h_j',\\
        \mu_j(B(\ubh)) & \text{if }\ h_j= h_j'.
    \end{cases}\]
\end{enumerate}
\end{prop}

\section{Comparison of matrices}
Recall that $I$ denotes the index set of a generalized Cartan matrix $C$. Let $[a,b]$ be an interval and $\mathbf{i}=(i_k)_{k\in [a,b]}$ a sequence of elements of $I$. In the following, we will consider chains of $\mathbf{i}$-boxes defined with respect to the sequence $\mathbf{i}$. Let $\mathfrak{C}=(\mathfrak{c}_k)_{k\in [1,l]}$ be a chain of $\mathbf{i}$-boxes on $[a,b]$ and let $(E_k)_{1\leq k<l}$ be the associated sequence of expansion operators. For any $1\leq k< l+1$, we write $\mathfrak{c}_{k}=[a_k,b_k]$. We associate to $\mathfrak{C}$ a signed word $\ubh(\mathfrak{C})=(\bh_k)_{1\leq k< l+1}$ as follows:

\begin{itemize}
    \item we set $\bh_1$ equal to the color $h_1$ of the $\mathbf{i}$-box $\mathfrak{c}_1$;
    \item for any $2\leq k<l+1$, we set $\bh_k=\varepsilon_k h_k$ where $h_k$ is the color of the $\mathbf{i}$-box $\mathfrak{c}_{k}$ and the sign $\varepsilon_k$ is defined by
    \[\begin{cases}
        1, &\text{if } E_{k-1}=L,\\
        -1,& \text{if } E_{k-1}=R.
    \end{cases}\]
\end{itemize}

It follows from our construction that we have $K(\mathfrak{C})=K(\ubh(\mathfrak{C}))$, $K(\mathfrak{C})^{\ufv}=K(\ubh(\mathfrak{C}))^{\ufv}$, and $K(\mathfrak{C})^{\fv}=K(\ubh(\mathfrak{C}))^{\fv}$, and the definition of $k[1]$ becomes identical. We will simply denote $K(\mathfrak{C})$ by $K$ below. Note that different chains of $\mathbf{i}$-boxes correspond to different signed words, but not every signed word comes from a chain of $\mathbf{i}$-boxes.

\begin{rem}\label{rem:translation}
Assume $k'[1] = k$ for some $k', k \in K = [1,\ell]$, and denote $k' = k[-1]$. Then, by our constructions, the following properties (1) to (6) are equivalent in each case:\vspace*{-0.8\baselineskip}
\begin{multicols}{2}
\begin{enumerate}
\item $\varepsilon_k = 1$
\item $E_{k-1} = L$
\item $a_k$ is the effective end of $[a_k, b_k]$
\item $\widetilde{a}_k = a_k$
\item $a_k = a_{k[-1]}^-$
\item $b_k = b_{k[-1]}$
\end{enumerate}

\columnbreak

\begin{enumerate}
\item $\varepsilon_k = -1$
\item $E_{k-1} = R$
\item $b_k$ is the effective end of $[a_k, b_k]$
\item $\widetilde{b}_k = b_k$
\item $b_k = b_{k[-1]}^+$
\item $a_k = a_{k[-1]}$
\end{enumerate}
\end{multicols}

\end{rem}

Assume that $b$ is an integer. In the following, we want to show that the matrices $B(\mathfrak{C})$ and $B(\ubh(\mathfrak{C}))$ are equal, thus providing a solution to Kashiwara--Kim--Oh--Park's problem.

To start with, consider the chain of $\mathbf{i}$-boxes $\mathfrak{C}_-^{[a,b]}$. Then the signed word $\ubh(\mathfrak{C}_-^{[a,b]})=(\bh_k)_{k\in [1,l]}$ is given by 
\[\bh_k=i_{b-k+1}.\]
Then the matrix $B(\ubh(\mathfrak{C}_-^{[a,b]}))=(\widetilde{b}_{jk})_{j\in K,k\in K^{\ufv}}$ simplifies to 

\[\widetilde{b}_{jk}=
\begin{cases}
    1 & k=j[1],\\
    -1 & j=k[1],\\
    c_{|\bh_{j}|,|\bh_{k}|}& j<k<j[1]<k[1],\\
    -c_{|\bh_{j}|,|\bh_{k}|}& k<j<k[1]<j[1].
\end{cases}
\]

Therefore, $B(\ubh(\mathfrak{C}_-^{[a,b]}))$ 
coincides with the matrix $B(\mathfrak{C}_-^{[a,b]})$.

\begin{prop}
\label{prop_box_move=flip}
Let $\mathfrak{C}=(\mathfrak{c}_k)_k$ be any chain of $\mathbf{i}$-boxes of range $[a,b]$. Suppose that the matrix $B(\mathfrak{C})$ equals $B(\ubh(\mathfrak{C}))$. Then for any movable $\mathbf{i}$-box $\mathfrak{c}_s$ of $\mathfrak{C}$, we also have $B(\nu_s(\mathfrak{C}))=B(\ubh(\nu_s(\mathfrak{C})))$.
\end{prop}

\begin{proof}
   Recall that we write $\ubh(\mathfrak{C})=(\bh_k)_k=(\varepsilon_k h_k)$. Denote the sequence of expansion operators associated to $\mathfrak{C}$ by $(E_k)_k$.
   
    Suppose that $s$ is greater than 1. Since $\mathfrak{c}_s$ is movable, we have $E_{s-1}\neq E_{s}$. Then, at the level of the associated signed word, we have $\varepsilon_{s}\neq \varepsilon_{s+1}$. Moreover, by Remark \ref{rem_color_exchange_op}, the signed word associated to $\nu_s(\mathfrak{C})$ is 
 $\ubh(\nu_s(\mathfrak{C}))=(\bh_k')_k,$ where
 \[\bh_k'=
\begin{cases}
    \bh_k & \text{if } k\notin \{s,s+1\}\\
    \bh_{s+1} & \text{if } k=s,\\
    \bh_s & \text{if } k=s+1.
\end{cases}
 \]
 In particular, the signed word $\ubh(\nu_s(\mathfrak{C}))$ is obtained from $\ubh(\mathfrak{C})$ through a flip move at $s$. We have the following two cases
 \begin{enumerate}
     \item If the $\mathbf{i}$-boxes $\mathfrak{c}_{s}$ and $\mathfrak{c}_{s+1}$ have the same color, then $|\bh_s|=|\bh_{s+1}|$. Therefore, the matrices $B(\ubh(\nu_s(\mathfrak{C})))$ and $B(\nu_s(\mathfrak{C}))$ are equal, since they are the mutation of the matrix $B(\ubh(\mathfrak{C}))=B(\mathfrak{C})$ at $s$.
     \item If the $\mathbf{i}$-boxes $\mathfrak{c}_{s}$ and $\mathfrak{c}_{s+1}$ have different color, then $|\bh_s|\neq|\bh_{s+1}|$. Therefore, the matrices $B(\ubh(\nu_s(\mathfrak{C})))$ and $B(\nu_s(\mathfrak{C}))$ are equal, since they are obtained from the matrix $B(\ubh(\mathfrak{C}))=B(\mathfrak{C})$ by applying $\sigma_s$.
 \end{enumerate}

If $s=1$, by Remark \ref{rem_color_exchange_op}, the signed word
 $\ubh(\nu_s(\mathfrak{C}))=(\bh_k')_k$ associated to $\nu_s(\mathfrak{C})$ is given by
 \[\bh_k'=
\begin{cases}
    \bh_k & \text{if } k\geq 3\\
    |\bh_{2}| & \text{if } k=1,\\
    -\varepsilon_2\bh_1 & \text{if } k=2.
\end{cases}
 \]
 Therefore, up to a left reflection which does not change the $B$-matrix, the signed word $\ubh(\nu_s(\mathfrak{C}))$ is obtained from $\ubh(\mathfrak{C})$ through a flip move, and we can conclude as above that $B(\nu_s(\mathfrak{C}))=B(\ubh(\nu_s(\mathfrak{C})))$.
\end{proof}

\begin{proof}[Proof of Theorem \ref{thm:box-to-word}]
     Recall that $\mathfrak{C}$ could be obtained from $\mathfrak{C}^{[a,b]}_-$ by finite many box moves, and $B(\mathfrak{C}^{[a,b]}_-)=B(\ubh(\mathfrak{C}^{[a,b]}_-))$. Applying Proposition \ref{prop_box_move=flip} repeatedly, we obtain $B(\mathfrak{C})=B(\ubh(\mathfrak{C}))$.
\end{proof}

\begin{cor}\label{cor:indepedent}
    The matrix $B(\mathfrak{C})$ does not depend on the choice of the chain transformation sending $\mathfrak{C}_-^{[a,b]}$ to $\mathfrak{C}$.
\end{cor}

Finally, assume that $b=+\infty$. Consider, for $s\geq 1$, the subchains $\mathfrak{C}_s=(\mathfrak{c}_k)_{1\leq k\leq s}$ of $\mathfrak{C}$. 
As a corollary of Proposition \ref{prop_box_move=flip}, the coefficients of the exchange matrices $B(\mathfrak{C}_s)$ stabilize:
\begin{cor}\label{cor:extension}
$\text{For any }s\geq 1, 
\text{ if } (i,j)\in K(\mathfrak{C}_s)\times K^{\text{ex}}(\mathfrak{C}_s), \text{ then } B(\mathfrak{C}_t)_{ij}=B(\mathfrak{C}_s)_{ij} \text{ for any }t\geq s.$
\end{cor}
Therefore, we can define the exchange matrix $B(\mathfrak{C})$ as the colimit of the matrices $B(\mathfrak{C}_s)$.
In other words, we have 
\[ B(\mathfrak{C})_{|K(\mathfrak{C}_s)\times K^{\text{ex}}(\mathfrak{C}_s)} = B(\mathfrak{C}_s) \ \ \ \text{for any } s\geq 1.\]

\subsection{Comparison of $B(\ubh(\mathfrak{C}))$ with Kashiwara--Kim's matrix}
\label{subsection_comparison_matrices}
We still write $[a,b]$ for an interval and $\mathfrak{C}=(\mathfrak{c}_k)_{k\in [1,l]}$ for a chain of $\mathbf{i}$-boxes $\mathfrak{c}_k=[a_k,b_k]$ of range $[a,b]$. For $1\leq k\leq l$, $h_k$ is the color of the $i$-box $\mathfrak{c}_k$.
Let $\widetilde{D}$ be the diagonal $l\times l$-matrix with diagonal entries $d_1=d_{h_1},\dots,d_l=d_{h_l}$, where the $(d_i)_{i\in I}$ are the diagonal entries of the minimal symmetrizer of the generalized Cartan matrix $C$. 
Following \cite{kashiwarakim2024exchangematricesiboxes}, 
we define $\widetilde{B}^\KK(\mathfrak{C})=(b^\KK_{kk'})_{k,k'\in K(\mathfrak{C})}$ as the skew-symmetrizable $l\times l$-matrix with skew-symmetrizer $\widetilde{D}$ and whose positive entries are given as follows:
\begin{equation}
\label{eq_entries_KK}
b^{\KK}_{j{k}}=
\begin{cases}
    1 & \text{if } (a_j=a_{k} \text{ and } b_{k}=b_j^-) \text{ or } (b_j=b_k \text{ and } a_{k}=a_j^-),\\
    -c_{h_j,h_{k}} &\text{if $c_{h_j,h_{k}}<0$ and one of the following conditions holds:}
\end{cases}
\end{equation}

\begin{itemize}
    \item[$(a)$]  $[a_j,b_j^+]\in\mathfrak{C},$ $a_j$ is the effective end of $[a_j,b_j]$, $a_{k}^-<a_j<a_{k}<b_{k}<b_j^+<b_{k}^+,$
    \item[$(b)$]  $[a_j,b_j^+]\in\mathfrak{C},$ $b_{k}$ is the effective end of $[a_{k},b_{k}]$, $a_k^-<a_j<b_j<b_{k}<b_j^+<b_{k}^+,$
    \item[$(c)$]  $[a_{k}^-,b_{k}]\in\mathfrak{C},$ $b_{k}$ is the effective end of $[a_k,b_k]$, $a_j^-<a_{k}^-<a_{j}<b_{j}<b_{k}<b_{j}^+,$
    \item[$(d)$]  $[a_{k}^-,b_{k}]\in\mathfrak{C},$ $a_j$ is the effective end of $[a_j,b_j]$, $a_j^-<a_{k}^-<a_{j}<a_{k}<b_{k}<b_j^+$.
\end{itemize}
Let $B^\KK(\mathfrak{C})$ denote the restriction of the matrix $\widetilde{B}(\mathfrak{C})$ to the indexes $K(\mathfrak{C})\times K^\mathrm{ex}(\mathfrak{C})$. When the Cartan matrix is symmetric of finite type and $\ubi$ is of the form $\widehat{\underline{w}}_0$, \cite{kashiwarakim2024exchangematricesiboxes} proved that $B^\KK(\mathfrak{C})=B(\mathfrak{C})$. Then Theorem \ref{thm:box-to-word} implies that $B^\KK(\mathfrak{C})=B(\ubh(\frC))$. In the next proposition, in the general case where $C$ is a generalized Cartan matrix, we compare them directly without using  Theorem \ref{thm:box-to-word}.

\begin{prop}
\label{prop_matrices_equal}
    The matrix $B^\KK(\mathfrak{C})$ is equal to the matrix $B(\ubh(\mathfrak{C}))$. 
\end{prop}
\begin{proof}
    Note that $\widetilde{B}^\KK(\mathfrak{C})$ and $\widetilde{B}(\ubh(\mathfrak{C}))$ have the same skew-symmetrizer. It suffices to prove that the positive $(j,k)$-entries are the same, where at least one of $j,k$ belongs to $K^\mathrm{ex}$.

    Let $j,k$ be in $[1,l]$ such that $c_{h_j,h_k}<0$. Recall from \eqref{eq_entries_KK} the conditions $(a),\ (b),\ (c)$ and $(d)$ which characterize such entries. Write $(i)$, $(ii)$, $(iii)$ and $(iv)$ for the conditions
\begin{itemize}
    \item[$(i)$]  $\varepsilon_{j[1]}=\varepsilon_{k}=-1,\ j<k<j[1]<k[1],$
\item[$(ii)$] $\varepsilon_{k}=-\varepsilon_{k[1]}=-1,\ j<k<k[1]<j[1]$,
\item[$(iii)$] $\varepsilon_{k[1]}=\varepsilon_{j}=1,\ k<j<k[1]<j[1]$,
\item[$(iv)$] $ \varepsilon_{j}=-\varepsilon_{j[1]}=1,\ k<j<j[1]<k[1].$
\end{itemize}
If any of these conditions is true, then we have $b_{jk}=-c_{h_j,h_k}$. Moreover, notice that the conditions $(i)$, $(ii)$, $(iii)$ and $(iv)$ are mutually exclusive, while $(a)$ (resp. $(c)$) and $(b)$ (resp. $(d)$) can hold at the same time.

\textbf{Step 1.}
First, assume that $(i)$ holds. Using Remark \ref{rem:translation}, we have that :
    \begin{itemize}
        \item $\varepsilon_k=-1$ is equivalent to $b_{k}$ being  the effective end of $[a_k,b_k]$.
        \item $\varepsilon_{j[1]}=-1$ is equivalent to  $[a_j,b_j^+]$ belonging  to $\mathfrak{C}$, i.e., $\frc_{j[1]}=[a_j,b_j^+]$.
        \item $j<k$ implies $[a_j,b_j]\subset [\ta_k,\ta_k]$. Then we deduce $a_k^-<\ta_k<a_j$.
        \item $j<k$ and $\varepsilon_k=-1$ imply that $b_j<\tb_k=b_k$.
        \item $k<j[1]$ and $\varepsilon_{j[1]}=-1$ imply that $b_k<\tb_{j[1]}=b_j^+$. 
        \item $j[1]<k[1]$ and $\varepsilon_{j[1]}=-1$ implies $b_j^+<b_k^+$. This claim follows from the fact that $b_j^+=b_{j[1]}\leq \tb_{k[1]}$ and from the construction of $\frc_{k[1]}$, which implies that $\tb_{k[1]}\leq b_k^+$. 
    \end{itemize} 
Therefore, $(i)$ implies $(b)$.

Now, assume that $(b)$ holds. Then we have the effective ends $b_k=\tb_k$ and $b_{j[1]}=\tb_{j[1]}=b_j^+$. Moreover, $\frc_{j[1]}=[a_j,b_j^+]$.  We claim that $(b)$ implies one of $(i)$-$(iv)$.

First, we cannot have $j[1]<k$ since $b_j^+=\tb_{j[1]}>\tb_k=b_k$. In addition, we cannot have $k[1]<j$. Otherwise, it would follow that $\frc_{k[1]}\subset [\ta_j,\tb_j]$, so that $b_{k[1]}\leq \tb_j< b_j^+$ and $a_{k[1]}\geq \ta_j> a_j^-$. But since we have either $b_{k[1]}=b_k^+$ or $a_{k[1]}=a_k^-$, both cases lead to contradiction with $(b)$.

If $j<k<j[1]<k[1]$, we are in the case $(i)$.

If $j<k<k[1]<j[1]$, we claim that $\varepsilon_{k[1]}=1$, i.e, we are in the case $(ii)$. To see this, assume $\varepsilon_{k[1]}=-1$. Then $\tb_{k[1]}=b_k^+>b_j^+=\tb_{j[1]}$, which contradicts $k[1]<j[1]$.

If $k<j<k[1]<j[1]$, as before, $k[1]<j[1]$ and $(b)$  imply $\varepsilon_{k[1]}=1$. We claim that $\varepsilon_j=1$, i.e,  we are in the case $(iii)$. To see this, note that $b_j<b_k$ but $j>k$, so $b_j$ is not the effective end of $\frc_j$, i.e, $\varepsilon_j=1$. The claim follows.

If $k<j<j[1]<k[1]$, as before, $j>k$ and $b_j<b_k$ imply $\varepsilon_j=1$. So we are in the case $(iv)$.

\textbf{Step 2.}
Similarly, assume that $(ii)$ holds. We have that:
    \begin{itemize}
        \item $\varepsilon_k=-1$ is equivalent to $b_{k}$ being the effective end of $[a_k,b_k]$.
        \item $\varepsilon_{k[1]}=1$ is equivalent to $[a_k^-,b_k]$ belonging to $\mathfrak{C}$, i.e., $\frc_{k[1]}=[a_k^-,b_k]$.
        \item $j<k$ implies $a_k^-<\ta_k<a_j$.
        \item $j<k$ and $\varepsilon_k=-1$ imply that $b_j<\tb_k=b_k$.
        \item $k[1]<j[1]$ and $\varepsilon_{k[1]}=1$ imply $b_k<b_j^+$ and $a_j^-<a_k^-$. This follows from the fact that $a_k^-=a_{k[1]}\geq \ta_{j[1]}\geq a_j^-$ and $b_k=b_{k[1]}\leq \tb_{j[1]}\leq b_j^+$.
    \end{itemize} 
Therefore, $(ii)$ implies $(c)$.

Now assume that $(c)$ holds. Then we have the effective ends $b_k=\tb_k$ and $a_{k[1]}=\ta_{k[1]}=a_k^-$ . Moreover, $\frc_{k[1]}=[a_k^-,b_k]$.  We claim that $(c)$ implies one of $(i)$-$(iv)$.

First, we cannot have $k[1]<j$. In fact, since $\frc_j=[a_j,b_j]\subset [a_k^-,b_k]=\frc_{k[1]}$, if $j>k[1]$, none of $a_j,b_j$ could be an effective end. In addition, we cannot have $j[1]<k$. Otherwise, it would follow that $\frc_{j[1]}\subset [\ta_k,\tb_k]$, thus $b_{j[1]}\leq \tb_k=b_k$ and $a_{j[1]}\geq \ta_k \pur{>} a_k^-$. Since either $b_{j[1]}=b_j^+$ or $a_{j[1]}=a_j^-$ holds, and both contradict $(c)$, we are led to a contradiction.

If $j<k<k[1]<j[1]$, we are in the case $(ii)$.

If $j<k<j[1]<k[1]$, we claim that $\varepsilon_{j[1]}=-1$, i.e, we are in the case $(i)$. To see this, assume $\varepsilon_{j[1]}=1$. Then $\ta_{j[1]}=a_j^-<a_k^-=\ta_{k[1]}$, which contradicts $j[1]<k[1]$.

If $k<j<j[1]<k[1]$, as before, $j[1]<k[1]$ and $(c)$ imply $\varepsilon_{j[1]}=-1$. Moreover, $b_j<b_k$ and $j>k$ imply that $b_j$ is not the effective end of $\frc_j$, i.e., $\varepsilon_j=1$. So we are in the case $(iv)$. 

If $k<j<k[1]<j[1]$, as before, $j>k$ and $b_j<b_k$ imply $\varepsilon_j=1$. So we are in the case $(iii)$.

\textbf{Step 3.}
Let us deduce the remaining cases from the previous steps. Let $\sigma$ denote the order reversing automorphism on $\Z$ such that $\sigma(x)=-x$. Consider the word $\textbf{i}'=(i'_k):=(i_{\sigma k})$ and the chain of $\textbf{i}'$-boxes $\frC'=(\frc'_s)_{s\in[1,l]}$ such that $\frc'_s:=[a'_s,b'_s]:=[\sigma b_s,\sigma a_s]$. Then, for any $j,k$ in $K$, the signed word $\ubh(\mathfrak{C})$ satisfies the conditions $(i),(ii),(iii)$, or $(iv)$ if and only if, respectively, the signed word $\ubh(\mathfrak{C}')=(\varepsilon_i'h_i)$ satisfies the conditions $(iii),(iv),(i),$ or $(ii)$ with $j$ and $k$ swapped:
\begin{itemize}
    \item[$(iii)$]  $\varepsilon_{j[1]}=\varepsilon_{k}=1,\ j<k<j[1]<k[1],$
\item[$(iv)$] $\varepsilon_{k}=-\varepsilon_{k[1]}=1,\ j<k<k[1]<j[1]$,
\item[$(i)$] $\varepsilon_{k[1]}=\varepsilon_{j}=-1,\ k<j<k[1]<j[1]$,
\item[$(ii)$] $ \varepsilon_{j}=-\varepsilon_{j[1]}=-1,\ k<j<j[1]<k[1].$
\end{itemize}
Similarly, for any $j,k$ in $K$, the chain of $\mathbf{i}$-boxes $\mathfrak{C}$ satisfies the conditions $(a),(b),(c)$, or $(d)$ if and only if, respectively, the chain of $\mathbf{i}$-boxes $\mathfrak{C}'$ satisfies the conditions $(c),(d),(a),$ or $(b)$ with $j$ and $k$ swapped:
\begin{itemize}
    \item[$(c)$]   $[(a')_j^-,(b')_j]\in\mathfrak{C}',$ $(b')_j$ is the effective end of $[(a')_j,(b')_j]$, $(b')_{k}^+ > (b')_j > (b')_{k} > (a')_{k}>(a')_j^->(a')_{k}^-,$
    \item[$(d)$]  $[(a')_j^-,(b')_j]\in\mathfrak{C}',$ $(a')_{k}$ is the effective end of $[(a')_{k},(b')_{k}]$, $(b')_k^+ > (b')_j > (a')_j>(a')_{k}>(a')_j^->(a')_{k}^-,$
    \item[$(a)$]  $[(a')_{k},b_{k}^+]\in\mathfrak{C}',$ $(a')_{k}$ is the effective end of $[(a')_k,(b')_k]$, $(b')_j^+>(b')_{k}^+>(b')_{j}>(a')_{j}>(a')_{k}>(a')_{j}^-,$
    \item[$(b)$]  $[(a')_{k},(b')_{k}^+]\in\mathfrak{C}',$ $(b')_j$ is the effective end of $[(a')_j,(b')_j]$, $(b')_j^+>(b')_{k}^+>(b')_{j}>(b')_{k}>(a')_{k}>(a')_j^-$.
\end{itemize}
Combining with the results in Steps 1,2, we obtain that $(iii)$ implies $(d)$, that $(iv)$ implies $(a)$ and that, if any among $(a)$, and $(d)$ holds, then we are in one of the cases $(i)$-$(iv)$.

Finally, we obtain the desired claim by comparing the explicit formula for the entries of the matrices, using the results of the previous steps. 
\end{proof}

\bibliographystyle{amsalphaURL}		
\bibliography{referenceEprint}    
\end{document}